\documentclass[a4paper]{amsart}
\usepackage{amssymb}
\usepackage{mathrsfs}

\theoremstyle{plain}
\newtheorem{theorem}{Theorem}[section]
\newtheorem{prop}[theorem]{Proposition}
\newtheorem{corollary}[theorem]{Corollary}
\newtheorem{lemma}[theorem]{Lemma}

\newtheorem{Theorem}{Main Theorem}

\theoremstyle{definition}
\newtheorem{remark}[theorem]{Remark}
\newtheorem{definition}[theorem]{Definition}
\newtheorem{example}[theorem]{Example}

\newcommand{\C}{\mathbb{C}}
\newcommand{\R}{\mathbb{R}}
\newcommand{\Z}{\mathbb{Z}}

\newcommand{\CP}{\mathbb{C}\mathrm{P}}

\newcommand{\CO}{\mathcal O}
\newcommand{\CL}{\mathcal L}

\renewcommand{\tilde}{\widetilde}
\renewcommand{\setminus}{\smallsetminus}
\newcommand{\nin}{/\kern-2.1ex\in}
\newcommand{\abs}[1]{\lvert#1\rvert}
\newcommand{\norm}[1]{\lVert#1\rVert}

\def\<{\left\langle}
\def\>{\right\rangle}

\def\ind{\operatorname{ind}}

\def\grad{\operatorname{grad}}
\def\Lie{\operatorname{Lie}}

\numberwithin{equation}{section}

\title[Equivariant local index in the symplectic cutting]{The equivariant local index of the reduced space in the symplectic cutting}
 \author[T. Yoshida]{Takahiko Yoshida}

\thanks{Partly supported by Grant-in-Aid for Scientific Research (C) 24540095.}

\address{Department of Mathematics, School of Science and Technology, Meiji University, 1-1-1 Higashimita, Tama-ku, Kawasaki, 214-8571, Japan}
\email{takahiko@meiji.ac.jp}

\begin{document}
\begin{abstract}
We compute the equivariant local index for the reduced space in a symplectic cut space, provided that the reduced space is compact. 
\end{abstract}
\maketitle

\section{Introduction}
In the joint work~\cite{FFY1,FFY2,FFY3} with Fujita and Furuta we developed an index theory for Dirac-type operators on possibly non-compact Riemannian manifolds provided that the end has an open covering $\{V_\alpha\}_{\alpha\in A}$ and each $V_\alpha$ is equipped with a torus fiber bundle structure $\pi_\alpha\colon V_\alpha\to U_\alpha$ and a Dirac-type operator $D_\alpha$ along fiber of $\pi_\alpha$ which satisfies certain acyclic condition. We call the index in our theory the {\it local index} and also call its equivariant version the {\it equivariant local index}. The local index has several properties, such as deformation invariance, excision property, and a normalization property that says the local index coincides with the analytic index of the Dirac-type operator when the manifold is closed. 

The theory of local index allows us to deal with the geometric quantizations of Lagrangian fibrations and that of Hamiltonian torus actions uniformly in the following sense. By applying the theory to a prequantizable Lagrangian fibration with closed total space we can show that the Riemann-Roch index is equal to the number of Bohr-Sommerfeld fibers and the sum of contributions of singular fibers~\cite{FFY1,FFY2}. In particular, when the Lagrangian fibration is nonsingular this implies that the Spin${}^c$ quantization is equivalent to the geometric quantization using the real polarization, at least in the level of the dimension of the quantum Hilbert space. 

The moment map of a Hamiltonian torus action can be thought of as a singular isotropic fibration. We can also apply the theory to this case and obtain the similar result. In particular, for a Hamiltonian $S^1$-action on a prequantizable closed symplectic manifold we can show that the equivariant Riemann-Roch index can be expressed as the sum of the equivariant local indices for the inverse images of the integer lattice points by the moment map. When the lattice point is a regular value of the moment map we can compute its equivariant local index and it is equal to the Riemann-Roch index of the symplectic quotient at the lattice point~\cite{FFY2,FFY3,Y8}. This gives a new proof of the Guillemin-Sternberg conjecture concerning the quantization commutes with the reduction for the Hamiltonian $S^1$ action. So the problem is how to compute the equivariant local index when the lattice point is a critical value. 

In this paper, as a special case of critical lattice points we compute the equivariant local index for the reduced space in a symplectic cut space. Let $(M,\omega)$ be a Hamiltonian $S^1$-manifold with moment map $\mu\colon M\to \R$. In this paper we identify the Lie algebra $\Lie (S^1)$ of $S^1$ with $\R$ by $\R\ni \theta \mapsto 2\pi\sqrt{-1}\theta\in \Lie(S^1)$, and also identify $\R^*$ with $\R$ by the standard inner product. Suppose that $n\in \R$ a regular value of $\mu$ and $S^1$ acts on the level set $\mu^{-1}(n)$ freely. We denote the reduced space $\mu^{-1}(n)/S^1$ by $M_n$. The symplectic cut is a procedure to make a new Hamiltonian $S^1$-manifold $\overline{M}_{\mu\le n}$, called the cut space, from the given $(M,\omega)$ and $n$. This method is introduced by Lerman in~\cite{Lerman} and the reduced space $M_n$ is naturally contained in the cut space $\overline{M}_{\mu\le n}$ as a fixed point set. In particular, $M_n$ is a singular fiber of the moment map of the $S^1$-action on $\overline{M}_{\mu\le n}$. 
Moreover, if the Hamiltonian $S^1$-manifold $(M,\omega)$ is prequantizable and the regular value $n$ is integer, then the cut space $\overline{M}_{\mu\le n}$ is also prequantizable, and a prequantum line bundle $(L,\nabla^L)$ on $(M,\omega)$ induces the prequantum line bundle $\left(\overline{L},\nabla^{\overline{L}}\right)$ on $\overline{M}_{\mu\le n}$. Suppose also that $\mu^{-1}(n)$ is compact. Let $O$ be an $S^1$-invariant open neighborhood of $M_n$ in $\overline{M}_{\mu\le n}$. Then the equivariant local index $\ind_{S^1}\left(O,O\setminus M_n;\overline{L}|_O\right)$ can be defined. The purpose of this paper is to give the following formula for $\ind_{S^1}\left(O,O\setminus M_n;\overline{L}|_O\right)$ which is Theorem~\ref{Main}. 
\begin{Theorem}
\[
\ind_{S^1}\left(O,O\setminus M_n;\overline{L}|_O\right)=\ind (M_n;L_n)\C_{(n)},
\]
where $L_n$ is the restriction of $\overline{L}$ to $M_n$, $\ind (M_n;L_n)$ is the Riemann-Roch index, and $\C_{(n)}$ is the irreducible representation of $S^1$ with weight $n$. 
\end{Theorem}
 This is an equivariant refinement of \cite[Theorem~5.4]{FFY4}, in which we give a formula for the local index for a reduced space in a symplectic cut space in order to determine the local index of the zero section of the cotangent bundle of the $n$-dimensional sphere $S^n$ by using the geodesic flow. 

 This paper is organized as follows. In Section~\ref{local index} we recall the equivariant local index in the case of the Hamiltonian $S^1$-actions. After that we briefly recall the symplectic cut in Section~\ref{symplectic cut}. The main theorem (Theorem~\ref{Main}) is given in Section~\ref{main theorem}. Finally examples are given in Section~\ref{examples}. 

\subsection*{Notation}
In this paper we use the notation $\C_{(n)}$ for the irreducible representation of $S^1$ with weight $n$. 

\section{Equivariant local index}\label{local index}
Let $(M,\omega)$ be a prequantizable Hamiltonian $S^1$-manifold and $(L,\nabla^L)$ an $S^1$-equivariant prequantum line bundle on $(M,\omega)$. 
We do not assume $M$ is compact. Since all orbits are isotropic the restriction of $(L,\nabla^L)$ to each orbit is flat. 
\begin{definition}
An orbit $\CO$ is said to be {\it $L$-acyclic} if $H^0\left(\CO; (L,\nabla^L)|_{\CO}\right)=0$. 
\end{definition}
 
 Let $V$ be an $S^1$-invariant open set whose complement is compact and which contains only $L$-acyclic orbits. For these data we give the following theorem. 
 \begin{theorem}[\cite{FFY1,FFY2,FFY3}]
There exists an element $\ind_{S^1}\left(M,V;L\right)\in R(S^1)$ of the representation ring such that $\ind_{S^1}\left(M,V;L\right)$ satisfies the following properties:
\begin{enumerate}
\item $\ind_{S^1}\left(M,V;L\right)$ is invariant under continuous deformation of the data. 
\item If $M$ is closed, then $\ind_{S^1}\left(M,V;L\right)$ is equal to the equivariant Riemann-Roch index $\ind_{S^1}\left(M;L\right)$. 
\item If $M'$ is an $S^1$-invariant open neighborhood of $M\setminus V$, then $\ind_{S^1}\left(M,V;L\right)$ satisfies the following excision property
\[
\ind_{S^1}\left(M,V;L\right)=\ind_{S^1}\left(M',M'\cap V;L|_{M'}\right).
\] 
\item $\ind_{S^1}\left(M,V;L\right)$ satisfies a product formula. 
\end{enumerate}
\end{theorem}
We call $\ind_{S^1}\left(M,V;L\right)$ the {\it equivariant local index}. 

\begin{example}\label{disc}
For a small positive real number $\varepsilon>0$ which is less than $1$ let $D_\varepsilon \left(\C_{(1)}\right)=\{ z\in \C_{(1)}\mid \abs{z}<\varepsilon\}$ be the $2$-dimensional disc of radius $\varepsilon$. As $(L,\nabla^L)\to (M,\omega)$ we consider 
\[
\left(D_\varepsilon \left(\C_{(1)}\right)\times \C_{(m)},d+\frac{1}{2}(zd\bar z-\bar zdz)\right)\to \left(D_\varepsilon \left(\C_{(1)}\right),\frac{\sqrt{-1}}{2\pi}dz\wedge d\bar z\right).
\]
First let us detect non $L$-acyclic orbits. Suppose the orbit $\CO$ through $z\in D_\varepsilon \left(\C_{(1)}\right)$ has a non-trivial parallel section $s\in H^0\left(\CO;(L,\nabla)|_{\CO}\right)$. Then $s$ satisfies the following equation
\[
0=\nabla^L_{\partial_\theta}s=\dfrac{\partial s}{\partial \theta}-2\pi\sqrt{-1}r^2s,
\]
where we use the polar coordinates $z=re^{2\pi\sqrt{-1}\theta}$. Hence $s$ is of the form $s(\theta)=s_0e^{2\pi\sqrt{-1}r^2\theta}$ for some non-zero constant $s_0$. Since $s$ is a global section on $\CO$ $s$ satisfies $s(0)=s(1)$. This implies $r=0$. 

Next, we put $V=D_\varepsilon \left(\C_{(1)}\right)\setminus \{0\}$ and let us compute $\ind_{S^1}(M,V;L)$. 
We recall the definition of $\ind_{S^1}(N,V;L)$. 
For $t\ge 0$ consider the following perturbation of the Spin${}^c$ Dirac operator $D\colon \Gamma\left(\wedge^{0,*} T^*M\otimes L\right)\to\Gamma\left(\wedge^{0,*} T^*M\otimes L\right) $ associated with the standard Hermitian structure on $M=D_\varepsilon \left(\C_{(1)}\right)$ 
\[
D_t=D+t\rho D_{S^1},
\]
where $\rho$ is a cut-off function of $V$ and $D_{S^1}$ is a first order formally self-adjoint differential operator of degree-one
\[
D_{S^1}\colon \Gamma\left((\wedge^\star T^*M^{0,1}\otimes L)|_V\right)\to \Gamma\left((\wedge^\star T^*M^{0,1}\otimes L)|_V\right)
\]
that satisfies the following conditions:
\begin{enumerate}
\item $D_{S^1}$ contains only derivatives along orbits. 
\item The restriction $D_{S^1}|_{\CO}$ to an orbit $\CO$ is the de Rham operator with coefficients in $L|_{\CO}$. 
\item For any $S^1$-equivariant section $u$ of the normal bundle $\nu_{\CO}$ of $\CO$ in $M$, $D_{S^1}$ anti-commutes with the Clifford multiplication of $u$. 
\end{enumerate}
See \cite{FFY1,FFY2,FFY3} for more details. From the second condition and $\{0\}$ is the unique non $L$-acyclic orbit we can see $\ker \left(D_{S^1}|_{\CO}\right)=0$ for all orbits $\CO\not=\{0\}$. Extend the complement of a neighborhood of $0$ in $D_\varepsilon \left(\C_{(1)}\right)$ cylindrically so that all the data are translationally invariant. Then we showed in \cite{FFY1,FFY2} that for a sufficiently large $t$ $D_t$ is Fredholm, namely, $\ker D_t\cap L^2$ is finite dimensional and its super-dimension is independent of a sufficiently large $t$. So we define 
\[
\ind_{S^1}\left(M,V;L\right)=\ker D_t^0\cap L^2-\ker D_t^1\cap L^2
\]  
for a sufficiently large $t$. In this case, by the direct computation using the Fourier expansion of $s$ with respect to $\theta$, 
we can show that 
\[
\ker D_t^0\cap L^2\cong \C, \ \ \ker D_t^1\cap L^2=0,
\]
and $\ker D_t^0\cap L^2$ is spanned by a certain $L^2$-function $a_0(r)$ on $D_\varepsilon \left(\C_{(1)}\right)$ which depends only on $r=\abs{z}$. Since the $S^1$-action on $\ker D_t^0\cap L^2$ is given by pull-back and the $S^1$-action on the fiber is given by $\C_{(m)}$ we obtain 
\[
\begin{split}
\ind_{S^1}\left(M,V;L\right)&=\ind_{S^1}\left(D_\varepsilon \left(\C_{(1)}\right),D_\varepsilon \left(\C_{(1)}\right)\setminus \{0\};D_\varepsilon \left(\C_{(1)}\right)\times \C_{(m)}\right)\\
&=\C_{(-m)}. 
\end{split}
\]
For more details see \cite[Remark~6.10]{FFY1}, or \cite[Section~5.3.2]{Y7}.  
\end{example}

It is well-known that the lift of $S^1$-action on $M$ to $L$ defines the moment map $\mu\colon M\to \R$ by the Kostant formula
\begin{equation}\label{Kostant formula}
\CL_Xs=\nabla^L_Xs+2\pi\sqrt{-1}\mu s,
\end{equation}
where $s$ is a section of $L$, $X$ is the vector field which generates the $S^1$-action on $(M,\omega)$, and $\CL_X s$ is the Lie derivative which is defined by
\[
\CL_Xs(x)=\dfrac{d}{d\theta}\Big|_{\theta=0}e^{-2\pi\sqrt{-1}\theta}s(e^{2\pi\sqrt{-1}\theta}x). 
\]
\begin{lemma}
If an orbit $\CO$ is not $L$-acyclic, namely, $H^0\left(\CO;(L,\nabla^L)|_{\CO}\right)\not=0$, then, $\mu(\CO)\in \Z$. 
\end{lemma}

If $M$ is closed, then we have the following localization formula for the equivariant Riemann-Roch index. 
\begin{corollary}
Suppose $M$ is closed. For $i\in \mu(M)\cap \Z$ let $V_i$ be an $S^1$-invariant open neighborhood of $\mu^{-1}(i)$ such that they are mutually disjoint, namely, $V_i\cap V_j\not=\emptyset$ for all $i\not= j$. Then, 
\begin{equation}\label{localization}
\ind_{S^1}\left(M;L\right)=\bigoplus_{i\in \mu(M)\cap \Z}\ind_{S^1}\left(V_i,V_i\cap V;L|_{V_i}\right) .
\end{equation}
\end{corollary}

\section{Symplectic cut}\label{symplectic cut}
In this section let us briefly recall the symplectic cut and its properties we use in this paper. See \cite{Lerman} for more details. Let $(M,\omega)$ be a Hamiltonian $S^1$-space with moment map $\mu\colon M\to \R$. For a real number $n$ the cut space $\overline{M}_{\mu\le n}$ of $(M,\omega)$ by the symplectic cutting~\cite{Lerman} is the reduced space of the diagonal $S^1$-action on $(M,\omega)\times \left(\C_{(1)},\frac{\sqrt{-1}}{2\pi}dz\wedge d\bar z\right)$, namely, 
\[
\overline{M}_{\mu\le n}=\left\{ \left. (x,z)\in (M,\omega)\times \left(\C_{(1)},\frac{\sqrt{-1}}{2\pi}dz\wedge d\bar z\right)\right| \mu(x)+\abs{z}^2=n\right\}/S^1.
\]
We denote the reduced space $\mu^{-1}(n)/S^1$ by $M_n$. 
\begin{prop}\label{prequantum condition on M}
\textup{(1)} If $S^1$ acts on $\mu^{-1}(n)$ freely, then, $\overline{M}_{\mu\le n}$ is a smooth Hamiltonian $S^1$-space. The $S^1$-action is given as 
\begin{equation}\label{actionMbar}
t[x,z]=[tx,z]
\end{equation}
for $t\in S^1$ and $[x,z]\in \overline{M}_{\mu\le n}$.\\
\textup{(2)} Under the assumption in $(1)$, the reduced space $M_n$ and $\{x\in M\mid \mu(x)< n\}$ are symplectically embedded into $\overline{M}_{\mu\le n}$ by $M_n\ni [x]\mapsto [x,0]\in \overline{M}_{\mu\le n}$ and $\{x\in M\mid \mu(x)< n\}\ni x\mapsto \left[x,\sqrt{n-\mu(x)}\right]\in \overline{M}_{\mu\le n}$, respectively. In particular, $\overline{M}_{\mu\le n}$ can be identified with the disjoint union $\{x\in M\mid \mu(x)< n\}\coprod M_n$ and with this identification $M_n$ is fixed by the $S^1$-action~\eqref{actionMbar}.
\end{prop}

Suppose that $(M,\omega)$ is equipped with a prequantum line bundle $(L,\nabla^L)\to (M,\omega)$ and the $S^1$-action lifts to $(L,\nabla^L)$ in such a way that $\mu$ satisfies the Kostant formula~\eqref{Kostant formula}. 
\begin{prop}\label{prequantum condition on L}
If $n$ is an integer and the $S^1$-action on $\mu^{-1}(n)$ is free, then $\overline{M}_{\mu\le n}$ is prequantizable. In this case a prequantum line bundle $(\overline{L},\nabla^{\overline{L}})$ on $\overline{M}_{\mu\le n}$ is given by 
\[
(\overline{L},\nabla^{\overline{L}})=\left((L,\nabla^L)\otimes \C_{(n)}\right)\boxtimes \left(\C_{(1)}\times \C_{(0)}, d+\frac{1}{2}(zd\bar z-\bar zdz)\right)\Big|_{\Phi^{-1}(0)}/S^1, 
\]
where $\Phi$ is the moment map $\Phi\colon M\times \C_{(1)}\to \R$ associated to the lift of the diagonal $S^1$-action which is written as $\Phi(x,z)=\mu(x)+\abs{z}^2-n$, and the lift of the $S^1$-action~\eqref{actionMbar} on $\overline{M}_{\mu\le n}$ to $(\overline{L},\nabla^{\overline{L}})$ is given by 
\begin{equation}\label{actionLbar}
t[u\otimes v\boxtimes (z,w)]=[(tu)\otimes v\boxtimes (z,w)]
\end{equation}
for $t\in S^1$ and $[u\otimes v\boxtimes (z,w)]\in \overline{L}$. The moment map $\overline{\mu}\colon \overline{M}_{\mu\le n}\to \R$ associated with the lift~\eqref{actionLbar} is written as $\overline{\mu}([x,z])=\mu(x)=n-\abs{z}^2$. 
\end{prop}
\begin{remark}\label{actionLn}
We denote the restriction of $(\overline{L},\nabla^{\overline{L}})$ to $M_n$ by $(L_n,\nabla^{L_n})$. $(L_n,\nabla^{L_n})$ is a prequantum line bundle on $M_n$. The $S^1$-action~\eqref{actionLbar} on $L_n$ is given by the fiberwise multiplication with weight $n$. Recall that $M_n$ is fixed by the $S^1$-action~\eqref{actionMbar}. See Proposition~\ref{prequantum condition on M}.
\end{remark}

\section{Main theorem}\label{main theorem}
Let $(M,\omega)$ be a prequantizable Hamiltonian $S^1$-manifold and $(L,\nabla^L)$ an $S^1$-equivariant prequantum line bundle on $(M,\omega)$ with the associated moment map $\mu\colon M\to \R$. Let $n$ be an integer and we assume the $S^1$-action on $\mu^{-1}(n)$ is free. Then, the cut space $\overline{M}_{\mu\le n}$ becomes a prequantizable Hamiltonian $S^1$-manifold and the $S^1$-equivariant prequantum line bundle $(\overline{L},\nabla^{\overline{L}})$ is given by Proposition~\ref{prequantum condition on M} and Proposition~\ref{prequantum condition on L}. 

Suppose that $\mu^{-1}(n)$ is compact. We take a sufficiently small $S^1$-invariant open neighborhood $O$ of $M_n$ in $\overline{M}_{\mu\le n}$ so that the intersection $\overline{\mu}(O)\cap \Z$ consists of the unique point $n$. Then we can define the equivariant local index $\ind_{S^1}\left(O,O\setminus M_n;\overline{L}|_O\right)$ of $M_n$ in $\overline{M}_{\mu\le n}$. We give the following formula for $\ind_{S^1}\left(O,O\setminus M_n;\overline{L}|_O\right)$. 
\begin{theorem}\label{Main}
Let $(M,\omega)$, $(L,\nabla^L)$, and $\mu$ be as above. Let $n$ be an integer. Suppose $S^1$ acts on $\mu^{-1}(n)$ freely and $\mu^{-1}(n)$ is compact. Let $O$ be a sufficiently small $S^1$-invariant open neighborhood of $M_n$ in $\overline{M}_{\mu\le n}$ which satisfies $\overline{\mu}(O)\cap \Z=\{ n\}$. Then, the equivariant local index is given as
\[
\ind_{S^1}\left(O,O\setminus M_n;\overline{L}|_O\right)=\ind(M_n;L_n)\C_{(n)},
\]
where 
$\ind(M_n;L_n)$ is the Riemann-Roch number of $M_n$. 
\end{theorem}
\begin{remark}
By replacing $\C_{(1)}$ with $\C_{(-1)}$ in the above construction we obtain the other cut space $\overline{M}_{\mu\ge n}=\{ (x,z)\in M\times \C_{-1}\colon \mu(x)-\abs{z}^2=n \}/S^1$. Theorem~\ref{Main} also holds for $\overline{M}_{\mu\ge n}$. 
\end{remark}

To prove the theorem we need some preliminaries. Let $\grad (\mu)$ be the negative gradient vector field of $\mu$ with respect to an $S^1$-invariant Riemannian metric $g$ of $M$, namely, the vector field determined by 
\[
-d\mu=g(\grad (\mu),\cdot\ ). 
\]
Since $n$ is a  regular value of $\mu$, $\grad (\mu)$ does not vanish near $\mu^{-1}(n)$. Let $\phi_\tau$ be the flow of the vector field $\frac{1}{\norm{\grad(\mu)}^2}\grad(\mu)$. Note that $\phi_\tau$ exists on $\mu^{-1}\left((n-\varepsilon,n+\varepsilon)\right)$ for a sufficiently small $\varepsilon>0$ since $\mu^{-1}(n)$ is compact by assumption. 
\begin{lemma}
$\phi_\tau$ commutes with the $S^1$-action on $M$ and satisfies the following property
\[
\mu\left(\phi_\tau(x)\right)=\mu(x)-\tau
\]
for $x\in \mu^{-1}\left((n-\varepsilon,n+\varepsilon)\right)$. 
\end{lemma}
\begin{proof}
\[
\begin{split}
\dfrac{d}{d\tau}\mu\left(\phi_\tau(x)\right)&=d\mu\left(\frac{1}{\norm{\grad(\mu)}^2}\grad(\mu)\right)\\
&=-g\left(\grad(\mu),\frac{1}{\norm{\grad(\mu)}^2}\grad(\mu)\right)\\
&=-1. 
\end{split}
\]
\end{proof}

By the definition of the symplectic cutting, the normal bundle $\nu$ of $M_n$ in $\overline{M}_{\mu\le n}$ is given by 
\[
\nu=\mu^{-1}(N)\times_{S^1}\C_{(1)}.
\]
For a sufficiently small $\varepsilon>0$ let $D_\varepsilon (\C_{(1)})=\{ z\in \C_{(1)}\colon \abs{z}<\varepsilon\}$ be the open disc of radius $\varepsilon$. We put $D_\varepsilon(\nu)=\mu^{-1}(n)\times_{S^1}D_\varepsilon(\C_{(1)})$, and define an $S^1$-action on $D_\varepsilon(\nu)$ by 
\begin{equation}\label{action on nu}
t[x,z]=[tx,z]. 
\end{equation}
Let $p\colon D_\varepsilon (\nu)\to M_n$ be the natural projection. We define a complex line bundle $L_{D_\varepsilon (\nu)}$ on $D_\varepsilon (\nu)$ by 
\[
L_{D_\varepsilon (\nu)}=p^*L_n\otimes \left(\mu^{-1}(n)\times_{S^1}(D_\varepsilon (\C_{(1)})\times \C_{(0)}) \right), 
\]
and define an lift of the $S^1$-action~\eqref{action on nu} to $L_{D_\varepsilon (\nu)}$ by 
\begin{equation}\label{action on nuL}
t\left(\left([x,z], [u\otimes v]\right)\otimes [x',z',w]\right)=\left(\left([tx,z],[(tu)\otimes v]\right)\otimes [tx',z',w]\right). 
\end{equation}
Then we can show the following lemma. 
\begin{lemma}\label{local model}
\textup{(1)} For a sufficiently small $\varepsilon>0$ there exists an $S^1$-equivariant embedding $f_M\colon D_\varepsilon (\nu)\to \overline{M}_{\mu\le n}$ with respect to the actions~\eqref{action on nu} and \eqref{actionMbar}. \\
\textup{(2)} Under the assumption in \textup{(1)} there exists an $S^1$-equivariant bundle map $f_L\colon L_{D_\varepsilon (\nu)}\to \overline{L}$ with respect to the actions~\eqref{action on nuL} and~\eqref{actionLbar} such that $f_L$ covers $f_M$.  
\end{lemma}
\begin{proof}
We define $f_M$ and $f_L$ by 
\[
\begin{split}
&f_M([x,z])=[\phi_{\abs{z}^2}(x),z],\\
&f_L(\left([x,z], [u\otimes v]\right)\otimes [x',z',w])=\left[\left(\tilde \phi_{\abs{z}^2}(t_0u)\otimes t_0^nv\right)\otimes (z,w)\right], 
\end{split}
\]
where $t_0\in S^1$ is the unique element that satisfies $t_0u\in L_x$ and $\tilde \phi_\tau(t_0u)$ is the horizontal lift of $\phi_\tau(x)$ starting from $t_0u$ and $\tilde \phi_{\abs{z}^2}(t_0u)$ is its value at $\tau=\abs{z}^2$. Then they are required maps. 
\end{proof}
%
%

\begin{proof}[Proof of Theorem~\ref{Main}]
By Lemma~\ref{local model} and the equivariant version of the product formula~\cite[Theorem~8.8]{FFY2} we obtain 
\begin{align}
&\ind_{S^1}\left(O,O\setminus M_n;\overline{L}|_O\right)\nonumber \\
&=\ind_{S^1}\left(D_\varepsilon (\nu),\ D_\varepsilon (\nu)\setminus M_n;L_{D_\varepsilon (\nu)}\right)\nonumber \\
&=\ind_{S^1}\left(M_n;L_n\otimes \mu^{-1}(n)\times_{S^1}\ind_{S^1}(D_\varepsilon(\C_{(1)}),D_\varepsilon(\C_{(1)})\setminus \{0\};D_\varepsilon(\C_{(1)})\times \C_{(0)})\right).
\end{align}
Note that the product formula for the $S^1$-equivariant local index holds since the $S^1$-action preserves all the data. See \cite[Section~6.2]{FFY3} for more details. From Example~\ref{disc} the equivariant local index $\ind_{S^1}(D_\varepsilon(\C_{(1)}),D_\varepsilon(\C_{(1)})\setminus \{0\};D_\varepsilon(\C_{(1)})\times \C_{(0)})$ is equal to $\C_{(0)}$. By definition, $L_n$ is naturally identified with the restriction of $\overline{L}$ to $M_n$. With this identification we can see that the restriction of the $S^1$-action~\eqref{actionLbar} to $L_n\to M_n$ is nothing but the fiberwise multiplication of $t^{-n}$. Since the $S^1$-action on $\ind_{S^1}(M_n;L_n)$ is defined by the pull-back, the $S^1$-action on $\ind_{S^1}(M_n;L_n)$ is given by the multiplication of $t^n$ as we mentioned in Remark~\ref{actionLn}. This proves the theorem. 
\end{proof}

\section{Examples}\label{examples}
\begin{example}[Complex projective space]\label{CPm}
As $(L,\nabla)\to (M,\omega)$ we adopt 
\[
\left((\C_{(1)})^m\times \C_{(0)},d+\frac{1}{2}\sum_{i=1}^m(z_id\bar z_i-\bar z_idz_i)\right)\to \left((\C_{(1)})^m,\frac{\sqrt{-1}}{2\pi}\sum_{i=1}^mdz_i\wedge d\bar z_i\right)
\]
equipped with the diagonal $S^1$-action on $(\C_{(1)})^m$ and its trivial lift to $(\C_{(1)})^m\times \C_{(0)}$. For $n=1$ the obtained $\overline{M}_{\mu\le n}$, $\overline{L}$, and $M_n$ are $\C P^m$, $\CO(1)$, and $\C P^{m-1}$, respectively. The induced $S^1$-actions on $\overline{M}_{\mu\le n}$ and $\overline{L}$ are given by
\begin{equation}\label{actionCPm}
\begin{split}
&t[z_1:\cdots :z_m:z_{m+1}]=[tz_1:\cdots :tz_m:z_{m+1}],\\
&t[z_1:\cdots :z_m:z_{m+1},w]=[tz_1:\cdots :tz_m:z_{m+1},w]. 
\end{split}
\end{equation}
The moment map $\overline{\mu}$ associated to the $S^1$-action~\eqref{actionCPm} is given by $\overline{\mu}([z_1:\cdots :z_{m+1}])=\sum_{i=1}^m\abs{z_i}^2$. For $k=0, 1$ let $O_k$ be a sufficient small $S^1$-invariant open neighborhood of $\overline{\mu}^{-1}(k)$. 
Then the equivariant local index $\ind_{S^1}\left(O_k,O_k\setminus \overline{\mu}^{-1}(k);\overline{L}|_{O_k}\right)$ is defined and By Corollary~\ref{localization} the equivariant Riemann-Roch index $\ind_{S^1}\left(\overline{M}_{\mu\le n},\overline{L}\right)$ satisfies following equality
\begin{equation}\label{localization formula}
\ind_{S^1}\left(\overline{M}_{\mu\le n},\overline{L}\right)=\ind_{S^1}\left(O_0,O_0\setminus \overline{\mu}^{-1}(0);\overline{L}|_{O_0}\right)+\ind_{S^1}\left(O_1,O_1\setminus \overline{\mu}^{-1}(1);\overline{L}|_{O_1}\right). 
\end{equation}

The left hand side is computed as
\begin{equation}\label{equiv index CPm}
\ind_{S^1}\left(\overline{M}_{\mu\le n},\overline{L}\right)=\ind_{S^1}\left(\C P^m,\CO(1)\right)=\C_{(0)}\oplus m\C_{(1)}. 
\end{equation}

For $k=1$, since $\overline{\mu}^{-1}(1)=M_n$, by Theorem~\ref{Main} 
$\ind_{S^1}\left(O_1,O_1\setminus \overline{\mu}^{-1}(1);\overline{L}|_{O_1}\right)$ is given as 
\begin{equation}\label{equiv index ON}
\begin{split}
\ind_{S^1}\left(O_1,O_1\setminus \overline{\mu}^{-1}(1);\overline{L}|_{O_1}\right)&=\ind_{S^1}\left(O_1,O_1\setminus M_n;\overline{L}|_{O_1}\right)\\
&=\ind(\CP^{m-1};\CO(1))\C_{(1)}\\
&=m\C_{(1)}. 
\end{split}
\end{equation}


For $k=0$, it is easy to see that $\overline{\mu}^{-1}(0)=\{ [z_0:0:\cdots :0] \}$ and $(\overline{L},\nabla^{\overline{L}})|_{[z_0:0:\cdots :0]}\cong (\C_{(0)},d+\frac{1}{2}(\bar zdz-zd\bar z))$. We can take $O_0$ in such a way that $O_0$ is identified with a sufficiently small open disc $D=\{ (z_1,\ldots ,z_m)\in \C^m\colon \sum_{i=1}^m\abs{z_i}^2\le \varepsilon \}$ with $S^1$-action $t(z_1,\ldots ,z_m)=(tz_1,\ldots ,tz_m)$. Then, by~\eqref{localization formula}, ~\eqref{equiv index CPm}, and ~\eqref{equiv index ON} 
we obtain the following formula 
\[
\ind_{S^1}\left(D,D\setminus \{ 0\};D\times \C_0\right)=\C_{(0)}. 
\]
In the case of $m=1$ this formula can be obtained in \cite[Remark~6.10]{FFY1} and \cite[Section~5.3.2]{Y7}.  
\end{example}

\begin{example}[Exceptional divisor]
Let $n$ and $(L,\nabla)\to (M,\omega)$ be as in Example~\ref{CPm}. Then the obtained cut space $\overline{M}_{\mu\ge n}$ is the blow-up $\tilde \C^m$ of the origin in $\C^m$, and $M_n$ and $L_n$ are the exceptional divisor $\C P^{m-1}$ and $\CO(n)$, respectively. We take a sufficiently small invariant open neighborhood $O$ of $M_n$. Then, by Theorem~\ref{Main} the equivariant local index $\ind_{S^1}\left(O,O\setminus M_n;\overline{L}|_O\right)$ is given by
\[
\ind_{S^1}\left(O,O\setminus M_n;\overline{L}|_O\right)=\ind\left(\CP^{m-1};\CO(n)\right)\C_{(n)}=\binom{m-1+n}{m-1}\C_{(n)}. 
\]
\end{example}
\bibliographystyle{amsplain}
\bibliography{references}
\end{document}